\documentclass[final,5p,times]{elsarticle}

\usepackage{amsmath,amssymb,amsfonts}
\usepackage{amsthm}
\usepackage{mathtools}

\usepackage{graphicx}
\graphicspath{{Figures/}}

\usepackage{booktabs}
\usepackage{multirow}
\usepackage{float}
\usepackage{subfig}

\usepackage{algorithm}
\usepackage{algpseudocode}

\usepackage{tikz}
\usetikzlibrary{
calc,
arrows,
automata,
backgrounds,
petri,
decorations.pathmorphing,
shapes,
shapes.geometric,
positioning
}

\usepackage[nameinlink]{cleveref}
\usepackage{placeins}

\newtheorem{theorem}{Theorem}[section]
\newtheorem{proposition}[theorem]{Proposition}

\newtheorem{definition}[theorem]{Definition}
\newtheorem{remark}[theorem]{Remark}
\newtheorem{assumption}[theorem]{Assumption}

\journal{Aerospace Science and Technology}

\begin{document}


\begin{frontmatter}


\title{
Geometry-Consistent Bayesian Filtering under Structural Model Uncertainty:
A Geometric Projection Particle Filter
}


\author[ursc]{Surya Ratna Prakash D}
\ead{dsrp@ursc.gov.in}

\author[iisc]{Soumyendu Raha\corref{cor1}}
\ead{raha@iisc.ac.in}

\cortext[cor1]{Corresponding author.}


\affiliation[ursc]{
organization={U R Rao Satellite Centre, Indian Space Research Organisation (ISRO)},
city={Bengaluru},
postcode={560017},
country={India}
}

\affiliation[iisc]{
organization={Centre for Computational and Data Sciences, Indian Institute of Science},
city={Bengaluru},
postcode={560012},
country={India}
}

\begin{abstract}

Nonlinear state estimation under structural model uncertainty remains a fundamental challenge in autonomous Guidance, Navigation, and Control (GNC) systems. Conventional Bayesian filtering separates state propagation from measurement correction, allowing model mismatch to accumulate during propagation and leading to proposal--likelihood inconsistency, particle degeneracy, and degraded estimation accuracy. Existing approaches primarily improve proposal distributions or weighting strategies without explicitly enforcing compatibility between system dynamics and measurement geometry during state evolution.

This paper introduces a geometry-consistent Bayesian filtering framework that incorporates measurement geometry directly into the propagation process. The nominal drift is projected onto the measurement-consistent subspace, yielding a geometry-consistent proposal while preserving the Bayesian posterior through a rigorous change-of-measure formulation. A Geometric Projection Particle Filter (GPF) is developed together with a geometric co-state that quantifies instantaneous dynamics--measurement inconsistency and provides an intrinsic indicator of estimator integrity.

Theoretical analysis establishes existence and uniqueness of the projected dynamics, posterior preservation, standard Monte Carlo convergence of the particle approximation, and robustness under structural model uncertainty by showing that the filtering error is governed by the component of model mismatch orthogonal to the measurement-consistent subspace.

The framework is validated using lunar descent navigation under partial observability and persistent structural model uncertainty. Compared with the bootstrap particle filter and conventional Gaussian filtering methods, GPF consistently maintains higher effective sample size and lower estimation error across representative operating conditions. These results demonstrate that geometry-consistent propagation provides a principled, computationally efficient, and theoretically grounded framework for robust nonlinear Bayesian filtering.
\end{abstract}

\begin{highlights}

\item Geometry-consistent proposal improves nonlinear particle filtering.

\item Bayesian posterior preservation is established via change of measure.

\item Geometric co-state quantifies structural model mismatch and consistency.

\item Filtering error is governed by projected mismatch orthogonal to measurement geometry.

\item Lunar descent experiments demonstrate improved accuracy and reduced degeneracy.

\end{highlights}


\end{frontmatter}

\section{Introduction}

Reliable nonlinear state estimation is fundamental to autonomous Guidance, Navigation, and Control (GNC) systems operating under uncertain and dynamically evolving environments. Applications such as autonomous spacecraft navigation, planetary landing, aerial robotics, and autonomous vehicles require estimators that remain statistically consistent despite structural model uncertainty, partial observability, and stochastic disturbances. Although Bayesian filtering provides a rigorous probabilistic framework for nonlinear state estimation, maintaining estimator consistency under persistent model mismatch remains a longstanding challenge.

Modern Bayesian filtering is founded on two complementary operations: state propagation using an assumed dynamical model followed by measurement-based posterior correction. This prediction--correction paradigm underlies classical Gaussian filters, including the Kalman Filter (KF), Extended Kalman Filter (EKF), and Unscented Kalman Filter (UKF), as well as Sequential Monte Carlo (SMC) methods such as the bootstrap Particle Filter (PF) \cite{Kalman1960,KalmanBucy1961,Maybeck1979,CrassidisJunkins2012,Doucet2001,Arulampalam2002,Kong1994,DoucCappe2005,Snyder2008}. While highly successful under accurate system models, the separation between propagation and correction allows structural model mismatch to accumulate continuously before measurement information is incorporated. The resulting inconsistency manifests as biased innovations in Gaussian filters and proposal--likelihood mismatch, particle degeneracy, and degraded estimation accuracy in particle-based methods.

Extensive research has therefore focused on improving proposal distributions and mitigating particle degeneracy within the Sequential Monte Carlo framework. Representative developments include optimal and auxiliary proposal distributions \cite{PittShephard1999,DoucetGodsillAndrieu2000}, implicit sampling \cite{ChorinMorzfeldTu2010,ChorinTu2015}, particle-flow methods \cite{DaumHuang2011}, ensemble transport approaches \cite{Reich2013,BunchGodsill2016,VanLeeuwen2019}, feedback particle filters \cite{MehtaMeyn2012}, and recent learning-based and score-based proposal mechanisms \cite{song2021score,levy2017robust,procope2026forward,zhang2025variational}. Additional developments include UKF-generated proposals \cite{rebollo2024symmetry}, Neural Particle Filters \cite{fang2024full}, and Rao--Blackwellized particle filters with adaptive noise estimation \cite{badar2024rao}. Although these methods improve proposal quality, sampling efficiency, or weight variance, they generally preserve the conventional prediction--correction architecture and therefore do not explicitly incorporate measurement geometry into the propagation process itself.

This observation motivates a different approach. Rather than improving proposal distributions after state propagation, the proposed framework incorporates measurement geometry directly into the propagation process while preserving the Bayesian posterior. By modifying only the proposal dynamics, the framework mitigates the effects of structural model uncertainty before posterior correction.

The remainder of the paper develops the proposed framework systematically. Section~\ref{sec:framework} introduces the Geometry-Consistent Filtering Principle. Section~\ref{sec:GPD} develops the continuous-time geometric formulation and derives the projected dynamics. Section~\ref{sec:bayesian} establishes posterior preservation through a change-of-measure formulation. Sections~\ref{sec:SGPF} and~6 derive the Geometric Projection Particle Filter (GPF) and its algorithmic realization, followed by experimental validation using autonomous lunar descent navigation under structural model uncertainty.
\section{Geometry-Consistent Bayesian Filtering Framework}
\label{sec:framework}

This section establishes the proposed Geometry-Consistent Bayesian Filtering Framework. The central idea is to incorporate measurement geometry directly into state propagation while preserving the Bayesian posterior. Consequently, Geometric compatibility governs proposal construction, whereas Bayesian inference governs posterior estimation, as formalized by the following principle.

\begin{definition}[Geometry-Consistent Filtering Principle]
A nonlinear Bayesian filter satisfies the \emph{Geometry-Consistent Filtering Principle} if the propagated state evolution remains locally compatible with the instantaneous measurement geometry while preserving the Bayesian posterior distribution. State propagation is therefore constrained by geometric compatibility, whereas posterior inference remains governed by Bayesian probability.
\end{definition}

The framework comprises four complementary components:

\begin{enumerate}

\item \textbf{Geometric Compatibility.}
The local relationship between the system dynamics and the measurement geometry defines the measurement-consistent tangent subspace.

\item \textbf{Geometric Evolution.}
The nominal dynamics are projected onto the measurement-consistent subspace, yielding a geometry-consistent proposal process.

\item \textbf{Probabilistic Consistency.}
Posterior preservation is established through a rigorous change-of-measure formulation.

\item \textbf{Computational Realization.}
The Geometric Projection Particle Filter (GPF) provides a Sequential Monte Carlo implementation of the proposed framework.

\end{enumerate}

Figure~\ref{fig:framework} summarizes the overall architecture and the interaction between geometric evolution, probabilistic consistency, and computational realization.

\begin{figure*}[t]
\centering
\includegraphics[width=0.85\textwidth]{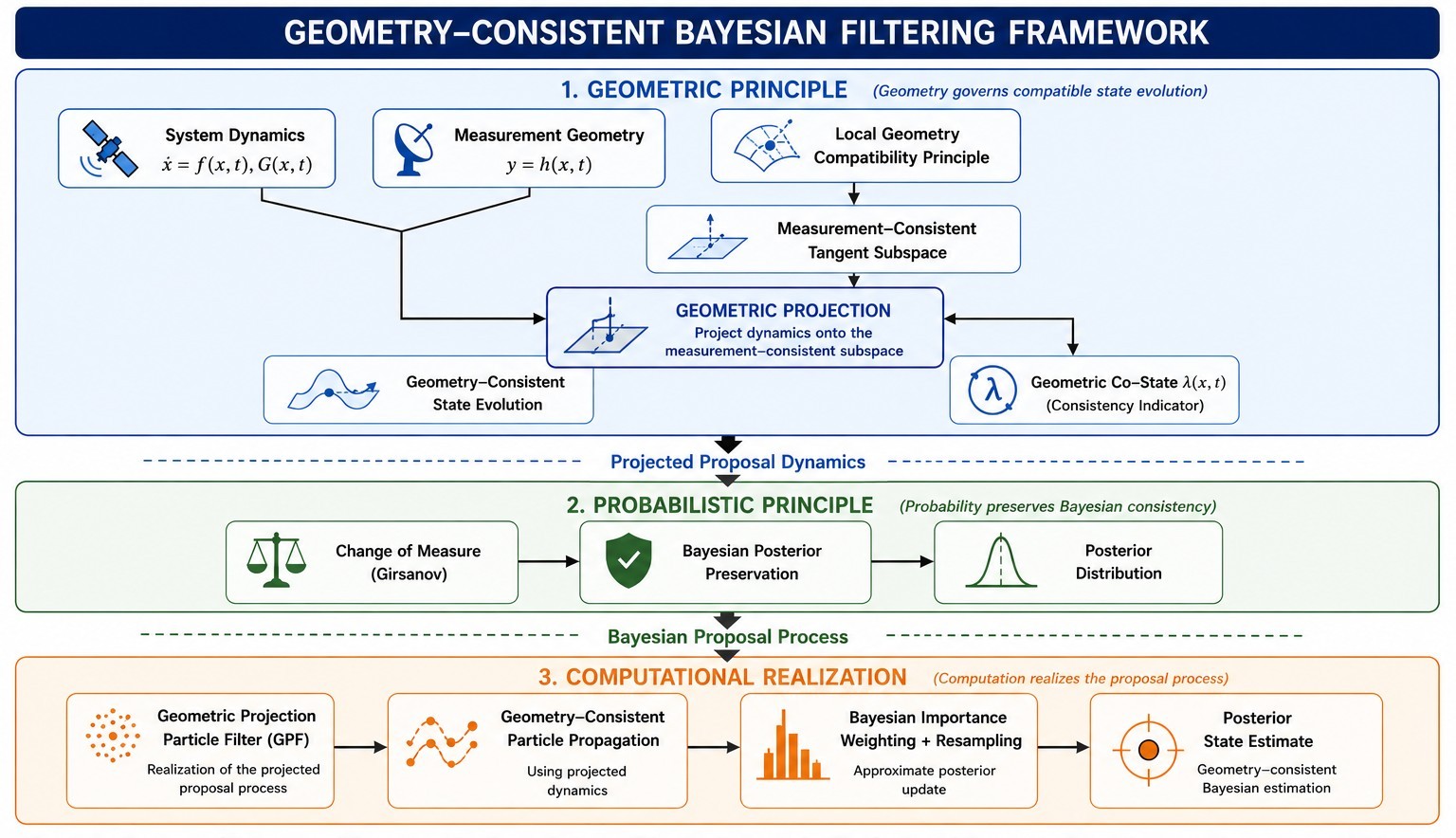}
\caption{
Architecture of the proposed Geometry--Consistent Bayesian Filtering Framework. Geometric evolution enforces measurement--dynamics compatibility through projected state propagation and the geometric co-state. Probabilistic consistency preserves the Bayesian posterior via a change-of-measure formulation. Computational realization implements the resulting proposal through the Geometric Projection Particle Filter (GPF).
}
\label{fig:framework}
\end{figure*}

The mathematical development proceeds in four stages. First, a continuous-time stochastic filtering model is established. Second, a local compatibility principle is introduced to characterize admissible state evolution. Third, the projected dynamics and the associated geometric co-state are derived. Finally, standard regularity assumptions establish the well-posedness of the resulting proposal process.

\section{Continuous-Time Geometric Framework}
\label{sec:GPD}
The framework is now formulated in continuous time. This section develops the geometric compatibility principle, projected dynamics, geometric co-state, and the corresponding well-posedness results.

\subsection{Continuous-Time Filtering Model}
\label{sec:filtering_model}

The Geometry--Consistent Bayesian Filtering Framework is formulated for nonlinear stochastic dynamical systems evolving in continuous time. Let
\[
(\Omega,\mathcal{F},\{\mathcal{F}_t\}_{t\ge0},\mathbb{P})
\]
be a filtered probability space satisfying the usual completeness and right-continuity conditions, where $\{\mathcal{F}_t\}$ denotes the natural filtration generated by the initial condition together with the driving Wiener processes.

The assumed system dynamics evolve according to the It\^o stochastic differential equation
\begin{equation}
dx_t
=
f(x_t,t)\,dt
+
G(x_t,t)\,dB_t,
\label{eq:state}
\end{equation}
where $x_t\in\mathbb{R}^{n}$ denotes the system state, $f:\mathbb{R}^{n}\times\mathbb{R}_{+}\rightarrow\mathbb{R}^{n}$ is the nominal drift field, $G:\mathbb{R}^{n}\times\mathbb{R}_{+}\rightarrow\mathbb{R}^{n\times r}$ is the diffusion matrix, and $B_t$ is an $r$-dimensional standard Wiener process.

The observation process is modeled by
\begin{equation}
y_t = h(x_t,t) + W_t,
\label{eq:obs}
\end{equation}
where $y_t\in\mathbb{R}^{m}$ denotes the measurement vector,
$h:\mathbb{R}^{n}\times\mathbb{R}_{+}\rightarrow\mathbb{R}^{m}$ is the observation map,
and $W_t$ is an $m$-dimensional Wiener process independent of $B_t$.
The initial state $x_0$ is assumed to be $\mathcal{F}_0$-measurable with a known prior distribution.

Equations~\eqref{eq:state} and~\eqref{eq:obs} define the nominal filtering model. The true system may deviate from these dynamics through structural model uncertainty. The following development constructs a geometry-consistent proposal process that preserves the Bayesian posterior while remaining locally compatible with the measurement geometry.

The subsequent development therefore focuses on the geometric structure of the filtering problem. In particular, the measurement Jacobian induces a local measurement-consistent tangent subspace that characterizes admissible state evolution. This geometric structure forms the basis for deriving the compatibility principle, geometric co-state, and projected dynamics introduced in the following subsections.

\subsection{Geometry Compatibility Principle}
\label{sec:compatibility}
The proposed framework enforces local compatibility between state evolution and the instantaneous measurement geometry by incorporating measurement information directly into the propagation process. Let $H_t := \frac{\partial h}{\partial x}(x_t,t)$ denote the measurement Jacobian evaluated along the system trajectory. The Jacobian defines the local tangent map from the state space to the observation space and therefore characterizes the instantaneous measurement geometry.

Applying a first-order differential approximation to the observation model yields

\begin{equation}
h(x_t+dx_t,t+dt)
= h(x_t,t) + H_t\,dx_t
+ \partial_t h(x_t,t)\,dt.
\end{equation}
where higher-order It\^o terms are neglected within the present first-order local compatibility approximation and are subsequently represented through the stochastic diffusion process.
The corresponding observation increment satisfies
\begin{equation}
dy_t
= H_t\,dx_t + \partial_t h(x_t,t)\,dt + dW_t.
\label{eq:obs_increment}
\end{equation}

Equation~\eqref{eq:obs_increment} establishes the differential relationship between state evolution and measurement evolution.

This motivates the following compatibility principle.

\begin{definition}[Geometry Compatibility Principle]
State evolution is said to be \emph{locally geometry-consistent} if the propagated state increment satisfies
\begin{equation}
H_t\,dx_t
+ \partial_t h(x_t,t)\,dt
= dy_t - dW_t.
\label{eq:compatibility}
\end{equation}
\end{definition}

Equation~\eqref{eq:compatibility} defines a local geometric constraint that characterizes admissible state evolution within the measurement-consistent tangent space. This compatibility relation forms the basis for the projected dynamics developed below.
\subsection{Reference Innovation}
\label{sec:innovation}

The compatibility principle introduced in the previous subsection establishes the local geometric relationship between state evolution and measurement evolution. To quantify deviations from this ideal compatibility, we introduce a reference innovation process that represents the measurement increment predicted by the nominal dynamics.

The deterministic component of the measurement evolution is defined as
\begin{equation}
\eta_t
:=
H_t f(x_t,t)
+
\partial_t h(x_t,t),
\label{eq:etadef}
\end{equation}
where the first term represents the contribution of the nominal system dynamics projected into the measurement space, while the second term accounts for the explicit time dependence of the observation map.

Using this quantity, the reference innovation increment is defined by
\begin{equation}
d\nu_t
:=
dy_t
-
\eta_t\,dt .
\label{eq:innovation}
\end{equation}

\begin{definition}[Reference Innovation]
The stochastic process
\[
d\nu_t
=
dy_t
-
\eta_t\,dt
\]
is called the \emph{reference innovation increment}. It represents the discrepancy between the observed measurement increment and the increment predicted by the nominal continuous-time dynamics.
\end{definition}

Substituting the observation model into \eqref{eq:innovation} gives $d\nu_t = dW_t$ under correct model specification. Consequently, the reference innovation reduces to the measurement noise increment and therefore constitutes a martingale with respect to the observation filtration.

Under structural model uncertainty, however, the deterministic component of the innovation generally deviates from zero. The resulting innovation therefore captures the instantaneous disagreement between the nominal dynamics and the measurement geometry, providing the fundamental quantity from which the proposed geometric correction is constructed. Unlike the classical innovation, the reference innovation governs proposal construction rather than posterior correction.
\subsection{Geometric Projection Dynamics}
\label{sec:projection}

The geometry compatibility principle established in the previous subsection defines the class of admissible state evolutions that remain locally consistent with the instantaneous measurement geometry. We now derive the continuous-time proposal dynamics satisfying this principle.

The compatibility condition
\begin{equation}
H_t\,dx_t
+
\partial_t h(x_t,t)\,dt
=
dy_t
-
dW_t
\label{eq:compatibility_repeat}
\end{equation}
generally cannot be satisfied by the nominal dynamics $dx_t=f(x_t,t)\,dt,$
since the nominal drift need not evolve within the measurement-consistent tangent subspace. Consequently, an additional correction is required to enforce local geometric compatibility.

To this end, we augment the nominal drift by introducing a correction term lying in the range of the measurement Jacobian,
\begin{equation}
dx_t
=
f(x_t,t)\,dt
+
H_t^{\top}\lambda_t\,dt,
\label{eq:modified_drift}
\end{equation}
where $\lambda_t\in\mathbb{R}^{m}$ is a geometric co-state that enforces the compatibility condition. Since $H_t^{\top}\lambda_t$ belongs to the measurement-consistent tangent space, the correction modifies only those components of the state evolution that directly influence the measurements.

Substituting
\eqref{eq:modified_drift}
into the compatibility relation
\eqref{eq:compatibility_repeat}
yields
\begin{equation}
H_tH_t^{\top}\lambda_t\,dt
=
dy_t
-
H_tf(x_t,t)\,dt
-
\partial_t h(x_t,t)\,dt
-
dW_t.
\label{eq:costate_system}
\end{equation}

Assuming the regularized Moore--Penrose pseudoinverse exists,
the corresponding co-state is
\begin{equation}
\lambda_t\,dt
=
(H_tH_t^{\top})^{\dagger}
\Big(
dy_t
-
H_tf(x_t,t)\,dt
-
\partial_t h(x_t,t)\,dt
-
dW_t
\Big).
\label{eq:lambda}
\end{equation}

Substituting
\eqref{eq:lambda}
into
\eqref{eq:modified_drift}
eliminates the auxiliary variable.

Under the first-order compatibility approximation established in
Section~\ref{sec:compatibility}, the resulting geometry-consistent proposal process is characterized by the following theorem.

\begin{theorem}[Geometry-Consistent Projection]
\label{thm:projection}
The geometry-consistent proposal process satisfying the Geometry Compatibility Principle is governed by the projected stochastic dynamics
\begin{equation}
\begin{aligned}
dx_t
&=
\Big(
I
-
H_t^{\top}
(H_tH_t^{\top})^{\dagger}
H_t
\Big)
f(x_t,t)\,dt
\\
&\quad
+
H_t^{\top}
(H_tH_t^{\top})^{\dagger}
\Big(
dy_t
-
\partial_t h(x_t,t)\,dt
-
dW_t
\Big).
\end{aligned}
\label{eq:proj}
\end{equation}
\end{theorem}

The projected dynamics admit the orthogonal decomposition.
Define
\begin{equation}
P_{\parallel}
=
H_t^{\top}
(H_tH_t^{\top})^{\dagger}
H_t,
\qquad
P_{\perp}
=
I-P_{\parallel},
\label{eq:projectors}
\end{equation}
where
$P_{\parallel}$
projects onto the measurement-consistent tangent subspace and
$P_{\perp}$
projects onto its orthogonal complement.

The projected dynamics can therefore be written compactly as
\begin{equation}
dx_t
=
P_{\perp}f(x_t,t)\,dt
+
P_{\parallel}
\Big(
dy_t
-
\partial_t h(x_t,t)\,dt
-
dW_t
\Big).
\label{eq:compact_proj}
\end{equation}

Equation~\eqref{eq:compact_proj} preserves the measurement-compatible component of the nominal dynamics while replacing the incompatible component through geometric projection.

\begin{remark}[Geometric Interpretation]
The projector
$P_{\parallel}$
maps the nominal dynamics onto the tangent space induced by the measurement Jacobian, whereas
$P_{\perp}$
removes the component orthogonal to this space.
Accordingly, the proposed dynamics evolve within the measurement-consistent tangent subspace while preserving the stochastic structure required for Bayesian filtering.
\end{remark}

\begin{remark}[Projection Scope]
The proposed projection modifies only the deterministic drift component of the proposal process, while the diffusion term remains unchanged. Consequently, the projected dynamics are interpreted as a proposal process adapted to the observation filtration. The subsequent change-of-measure formulation establishes probabilistic equivalence with the original filtering problem and preserves the Bayesian posterior.
\end{remark}

\begin{remark}[Numerical Implementation]
Throughout this work,
the Moore--Penrose pseudoinverse is evaluated using the Tikhonov-regularized approximation $(H_tH_t^{\top}+\epsilon I)^{-1},$ where $\epsilon>0$
ensures numerical robustness under rank-deficient or ill-conditioned measurement configurations.
\end{remark}

\subsection{Geometric Co-State}
\label{sec:costate}
The projected dynamics are generated through the auxiliary variable
$\lambda_t$, which arises naturally from the Geometry Compatibility Principle. The geometric co-state quantifies the instantaneous discrepancy between the nominal system dynamics and the local measurement geometry, thereby providing an intrinsic measure of dynamics--measurement inconsistency.

From~\eqref{eq:lambda}, the geometric co-state is given by
\begin{equation}
\lambda_t\,dt
= (H_tH_t^\top)^\dagger
\left( dy_t - H_tf(x_t,t)\,dt
- \partial_t h(x_t,t)\,dt - dW_t \right).
\label{eq:costate}
\end{equation}

Equation~\eqref{eq:costate} shows that
$\lambda_t$
is proportional to the component of the measurement increment that cannot be explained by the nominal continuous-time dynamics. Unlike the classical innovation, which is used for posterior correction, the geometric co-state characterizes proposal consistency before Bayesian weighting and therefore serves as an intrinsic indicator of estimator consistency.

The following result establishes the nominal behavior of the geometric co-state.
\begin{theorem}[Nominal Consistency of the Geometric Co-State]
Assume that the assumed system dynamics coincide with the true dynamics and that the observation model is correctly specified. Then
\[
\mathbb{E}
\!\left[
dy_t
\mid
\mathcal{F}_t^y
\right]
=
\eta(x_t,t)\,dt,
\]
which implies $ \mathbb{E} \!\left[
\lambda_t
\mid
\mathcal{F}_t^y
\right]
=
0.
$
\label{thm:costate_zero}
\end{theorem}

\begin{proof}

Under correct model specification,

\[
dy_t
=
\eta(x_t,t)\,dt
+
dW_t,
\]

where
$dW_t$
is a martingale increment with respect to the observation filtration.
Substituting this expression into
\eqref{eq:costate}
yields

\[
\lambda_t\,dt
=
(H_tH_t^\top)^\dagger dW_t.
\]

Taking conditional expectation with respect to
$\mathcal{F}_t^y$
and using

\[
\mathbb{E}[dW_t|\mathcal{F}_t^y]=0
\]

gives

\[
\mathbb{E}
[
\lambda_t
|
\mathcal{F}_t^y
]
=
0,
\]

which proves the result.

\end{proof}

Theorem~\ref{thm:costate_zero}
establishes that the geometric co-state is unbiased under correct model specification. Consequently, persistent non-zero values of
$\lambda_t$
indicate structural disagreement between the assumed dynamics and the measurement geometry.

This property enables the geometric co-state to serve as an intrinsic indicator of estimator consistency. Large co-state magnitudes correspond to increased dynamics--measurement incompatibility and therefore identify operating regimes in which structural model mismatch is expected to degrade filtering performance.

\begin{remark}[Discrete-Time Approximation]

For numerical implementation,
the continuous-time co-state admits the Euler approximation

\begin{equation}
\begin{aligned}
\lambda_k\,\Delta t
\approx
(H_kH_k^\top)^\dagger
\Big(
&
\Delta y_k
-
H_kf(x_k,t_k)\Delta t
\\
&
-
\partial_t h(x_k,t_k)\Delta t
-
\Delta W_k
\Big),
\end{aligned}
\label{eq:discrete_costate}
\end{equation}

which converges to the continuous-time expression as
$\Delta t\rightarrow0$.
The same quantity is computed for every particle within the proposed Geometric Projection Particle Filter and is subsequently used as a diagnostic indicator of geometric consistency.
\end{remark}

\subsection{Regularity Assumptions and Well-Posedness}
\label{sec:wellposed}

The proposed framework is analyzed under standard regularity assumptions for nonlinear stochastic filtering, ensuring that both the nominal dynamics and the projected proposal process are mathematically well defined.

\begin{assumption}[Dynamics Regularity]
\label{ass:regularity}
The functions $f:\mathbb{R}^{n}\times\mathbb{R}_{+}\rightarrow\mathbb{R}^{n},
G:\mathbb{R}^{n}\times\mathbb{R}_{+}\rightarrow\mathbb{R}^{n\times r},
$ and $ h:\mathbb{R}^{n}\times\mathbb{R}_{+}\rightarrow\mathbb{R}^{m}$ are continuously differentiable with respect to the state variable, jointly locally Lipschitz in $(x,t)$, and satisfy linear growth conditions.
\end{assumption}

\begin{assumption}[Measurement Jacobian]
\label{ass:jacobian}
The measurement Jacobian is progressively measurable and locally bounded on every compact time interval.
\end{assumption}

\begin{assumption}[Regularity of the Projection Operator]
\label{ass:projection}
The matrix $H_tH_t^{\top}$ admits a bounded regularized Moore--Penrose pseudoinverse $(H_tH_t^{\top})^{\dagger},$ which varies locally Lipschitz continuously with respect to the state.
\end{assumption}

These assumptions are standard in continuous-time nonlinear filtering and ensure that the measurement-consistent tangent subspace and its associated projection operator remain well defined throughout the evolution.

\begin{remark}[Partial Observability]
The formulation naturally accommodates partial and time-varying observability. When the measurement Jacobian is full row rank, the pseudoinverse reduces to the ordinary inverse; otherwise, the regularized Moore--Penrose pseudoinverse provides a numerically stable projection onto the measurement-consistent tangent subspace.
\end{remark}

\begin{theorem}[Existence and Uniqueness of the Geometry-Consistent Proposal]
\label{thm:existence}

Suppose Assumptions
\ref{ass:regularity}--\ref{ass:projection}
hold.

Then the projected stochastic differential equation
\eqref{eq:proj}
admits a unique strong solution on every finite time interval
$[0,T]$.

\end{theorem}

\begin{proof}

Under Assumption~\ref{ass:regularity},
the nominal drift
$f(x,t)$
and diffusion matrix
$G(x,t)$
are locally Lipschitz and satisfy linear growth conditions.

Assumptions
\ref{ass:jacobian}
and
\ref{ass:projection}
imply that the projection operator

\[
P_{\parallel}
=
H_t^{\top}
(H_tH_t^{\top})^{\dagger}
H_t
\]

is locally Lipschitz with respect to the state.

Consequently, the projected drift appearing in
\eqref{eq:proj}
inherits the same regularity properties as the nominal drift.

Since the diffusion term remains unchanged,$G(x_t,t)dB_t,$ the projected dynamics satisfy the standard existence and uniqueness conditions for It\^o stochastic differential equations. Therefore, classical SDE theory guarantees the existence of a unique strong solution on every finite time interval.
\end{proof}

Theorem~\ref{thm:existence}
establishes that the proposed geometric projection defines a mathematically well-posed continuous-time proposal process. Consequently, the geometry-consistent evolution introduced in this section provides a rigorous stochastic foundation for the Bayesian change-of-measure formulation developed in the following section.

The developments in this section establish the deterministic geometric structure underlying the proposed framework. The subsequent section shows that this geometry-consistent proposal preserves the Bayesian posterior through a rigorous change-of-measure formulation, thereby linking geometric consistency with probabilistic consistency.

\section{Bayesian Consistency of Geometry--Consistent Filtering}
\label{sec:bayesian}
The previous section derived a geometry-consistent proposal process whose state evolution remains locally compatible with the instantaneous measurement geometry. The remaining question is whether this modified proposal preserves the Bayesian posterior distribution. This section answers that question by establishing a change-of-measure formulation, deriving posterior preservation through the Kallianpur--Striebel formula, and developing the corresponding particle approximation \cite{xiong2008}.
\subsection{Change of Measure and Posterior Preservation}
\label{sec:change_measure}
The continuous-time geometric framework developed in Section~\ref{sec:GPD} establishes a geometry-consistent proposal process whose state evolution remains locally compatible with the instantaneous measurement geometry. Although this modified propagation improves geometric consistency, it does not by itself guarantee preservation of the Bayesian posterior distribution. It is therefore necessary to establish that the proposed geometry-consistent evolution remains probabilistically equivalent to the original nonlinear filtering problem.

The proposed framework modifies only the proposal dynamics while leaving the physical observation model unchanged. Consequently, the geometry-consistent proposal process remains adapted to the observation filtration
$\mathcal{F}_t^y$
and is therefore non-anticipative. This permits the construction of an equivalent probability measure relating geometry-consistent proposal process to the physical filtering process through Girsanov's theorem.

Let
$\mathbb{P}$
denote the physical probability measure associated with the original nonlinear filtering problem, under which the observation process satisfies~\eqref{eq:obs}. Define the likelihood (Radon--Nikodym) process

\begin{equation}
\mathcal{L}_t
=
\exp\!\left(
\int_0^t
\eta(x_s,s)^\top
dy_s
-
\frac12
\int_0^t
\|\eta(x_s,s)\|^2
ds
\right),
\label{eq:RN}
\end{equation}

where $\eta(x,t)$ is the deterministic measurement drift introduced in~\eqref{eq:etadef}. Since $\eta(x,t)$ coincides with the deterministic component of the reference innovation process defined in~\eqref{eq:innovation}, the likelihood process remains fully consistent with the geometry-consistent proposal dynamics.

Under
$\mathbb{P}$,
the observation process satisfies

\[
dy_t
=
\eta(x_t,t)\,dt
+
dW_t,
\]

where $W_t$ is an $m$-dimensional standard Wiener process.

The following assumption guarantees that the likelihood process defines a valid exponential martingale.

\begin{assumption}[Novikov Condition]
\label{ass:novikov}

The measurement drift $\eta(x,t)$ satisfies

\[
\mathbb{E}_{\mathbb{P}}
\left[
\exp
\left(
\frac12
\int_0^T
\|\eta(x_s,s)\|^2
ds
\right)
\right]
<
\infty.
\]

\end{assumption}

Under Assumption~\ref{ass:novikov},
the likelihood process
$\mathcal{L}_t$
defines the Radon--Nikodym derivative

\begin{equation}
\frac{d\mathbb{Q}}{d\mathbb{P}}
=
\mathcal{L}_T,
\label{eq:RN_derivative}
\end{equation}

which induces a probability measure
$\mathbb{Q}$
equivalent to
$\mathbb{P}$.

Consequently,

\begin{equation}
\tilde W_t
=
y_t
-
\int_0^t
\eta(x_s,s)\,ds
\label{eq:innovation_brownian}
\end{equation}

is an $m$-dimensional standard Wiener process under $\mathbb{Q}$..

The following theorem establishes that the geometry-consistent proposal preserves the Bayesian posterior distribution.

\begin{theorem}[Posterior Preservation under the Geometry-Consistent Proposal]
\label{thm:posterior}

Suppose Assumption~\ref{ass:novikov} holds together with the regularity assumptions of Section~\ref{sec:GPD}. Then geometry-consistent proposal process~\eqref{eq:proj} defines a well-posed It\^o diffusion adapted to the observation filtration. Moreover, the Bayesian posterior conditioned on
$\mathcal{F}_t^y$
is preserved through the Kallianpur--Striebel representation,

\begin{equation}
\mathbb{E}_{\mathbb{P}}
\!\left[
u(x_t)
\mid
\mathcal{F}_t^y
\right]
=
\frac{
\mathbb{E}_{\mathbb{Q}}
\!\left[
u(x_t)\,
\mathcal{L}_t
\right]
}{
\mathbb{E}_{\mathbb{Q}}
\!\left[
\mathcal{L}_t
\right]
},
\label{eq:KS}
\end{equation}

for every bounded measurable test function
$u$.

\end{theorem}

\begin{proof}

Assumption~\ref{ass:novikov} ensures that
$\mathcal{L}_t$
is a true martingale and therefore defines an equivalent probability measure through~\eqref{eq:RN_derivative}. By Girsanov's theorem, the transformed observation process~\eqref{eq:innovation_brownian} is a standard Wiener process under $\mathbb{Q}$.

The geometry-consistent proposal process developed in Section~\ref{sec:GPD} remains adapted to the observation filtration and satisfies the regularity assumptions required for existence and uniqueness of strong solutions. Consequently, it defines a well-posed non-anticipative diffusion under both probability measures.

The Bayesian posterior is then recovered through the classical Kallianpur--Striebel representation~\eqref{eq:KS}. Therefore, the proposed geometry-consistent propagation modifies only the proposal dynamics while preserving the Bayesian posterior distribution.
\end{proof}
\subsection{Particle Approximation}
\label{sec:particle}

The posterior preservation result established in the previous subsection guarantees that the geometry-consistent proposal process generates the correct Bayesian posterior under importance weighting. We now establish the corresponding Sequential Monte Carlo approximation.

Let $ \{x_t^{(i)}\}_{i=1}^{N}$ denote particles evolving independently according to the geometry-consistent proposal dynamics~\eqref{eq:proj}. The associated empirical measure is

\begin{equation}
\pi_t^{N}
:=
\frac{1}{N}
\sum_{i=1}^{N}
\delta_{x_t^{(i)}},
\label{eq:empirical_measure}
\end{equation}

where
$\delta_x$
denotes the Dirac measure concentrated at
$x$.

Each particle is assigned an importance weight proportional to the likelihood process,

\[
w_t^{(i)}
\propto
\mathcal{L}_t^{(i)},
\]

where
$\mathcal{L}_t^{(i)}$
is the likelihood process~\eqref{eq:RN}
evaluated along the trajectory of the
$i$-th particle.

The corresponding self-normalized particle approximation of a bounded measurable function
$u$
is

\begin{equation}
\Pi_t^{N}(u)
=
\frac{
\sum_{i=1}^{N}
u(x_t^{(i)})\,w_t^{(i)}
}{
\sum_{i=1}^{N}
w_t^{(i)}
}.
\label{eq:particle_estimator}
\end{equation}

The following theorem establishes convergence of the proposed particle approximation.

\begin{theorem}[Consistency of the Particle Approximation]
\label{thm:particle_consistency}

Suppose the assumptions of
Theorem~\ref{thm:posterior}
hold together with the standard assumptions for self-normalized Sequential Monte Carlo methods, including finite second moments and non-degenerate importance weights.

Then, for every bounded measurable test function
$u$,

\[
\Pi_t^{N}(u)
\xrightarrow{\mathbb{P}}
\mathbb{E}_{\mathbb{P}}
\!\left[
u(x_t)
\mid
\mathcal{F}_t^{y}
\right]\, as \  N\rightarrow\infty.\]

Furthermore,
\[
\mathbb{E}
\!\left[
\left| \Pi_t^{N}(u)
- \mathbb{E}_{\mathbb{P}}
\!\left[ u(x_t) \mid
\mathcal{F}_t^{y}
\right]
\right|^{2} \right]^{1/2} =
\mathcal{O}(N^{-1/2}).
\]

\end{theorem}

\begin{proof}

Theorem~\ref{thm:posterior}
establishes that the geometry-consistent proposal preserves the Bayesian posterior distribution through the likelihood process. Consequently, the proposed estimator is a standard self-normalized importance sampling estimator.

The convergence result therefore follows directly from classical Sequential Monte Carlo theory using propagation-of-chaos arguments and the asymptotic consistency of self-normalized particle approximations~\cite{delmoral2004feynman}.

\end{proof}

\paragraph*{Weight Stability.}

Unlike the bootstrap particle filter, the proposed geometry-consistent proposal aligns particle propagation with the local measurement geometry before likelihood evaluation. Consequently, the mismatch between proposal and posterior distributions is reduced, leading to lower importance-weight dispersion and improved effective sample size. This theoretical behavior is verified experimentally in Section~\ref{sec:results}.

\subsection{Robustness and Error Decomposition}
\label{sec:robustness_theory}

The preceding analysis established that the proposed geometry-consistent proposal preserves the Bayesian posterior and admits a consistent particle approximation. We now characterize the effect of structural model mismatch on the resulting estimation error.

Suppose that the true system drift is

\[
f_{\mathrm{true}}(x,t)
= f(x,t) + \Delta f(x,t),
\]

where $\Delta f$ denotes an unknown model perturbation.

Since the proposed projection removes the component of the drift incompatible with the measurement geometry, only the residual mismatch lying outside the measurement-consistent tangent subspace contributes directly to the estimation error.

The following result formalizes this robustness property.

\begin{proposition}[Robustness under Structural Model Mismatch]
\label{prop:robustness}

Suppose the assumptions of Theorems~\ref{thm:projection} and~\ref{thm:particle_consistency}
hold. Assume further that the model perturbation $\Delta f(x,t)$
is locally integrable on every finite interval $[0,T]$.
Then the estimation error induced by structural model mismatch satisfies

\begin{equation}
\left\| e(t) \right\|
\le \int_0^t \left\|
(I-P_{\parallel}(x_s))
\Delta f(x_s,s) \right\|
ds + \int_0^t \| \lambda_s \| ds,
\label{eq:robustness}
\end{equation}

where
$P_{\parallel}$
is the measurement-consistent projection operator and
$\lambda_t$
is the geometric co-state introduced in Section~\ref{sec:costate}.

\end{proposition}

\begin{proof}
Let $e(t):=x_{\mathrm{true}}(t)-x(t)$ denote the estimation error between the true system and the geometry-consistent proposal process. Assuming identical initial conditions,$e(0)=0.$

Under structural model mismatch, $f_{\mathrm{true}} = f+\Delta f,$ the projected dynamics remove the component of the perturbation aligned with the measurement-consistent tangent subspace. Consequently, the estimation error evolves according to

\[
\dot e(t) = (I-P_{\parallel}(x_t))
\Delta f(x_t,t)
+ \lambda_t,
\]

where $\lambda_t$ represents the residual dynamics--measurement incompatibility quantified by the geometric co-state.
Integrating over the interval $[0,t]$ gives
\[
e(t) = \int_0^t (I-P_{\parallel}(x_s))
\Delta f(x_s,s)\,ds
+ \int_0^t \lambda_s\,ds.
\]

Applying the triangle inequality immediately yields

\[
\|e(t)\| \le \int_0^t
\| (I-P_{\parallel}(x_s))
\Delta f(x_s,s) \| ds + \int_0^t \| \lambda_s \| ds,
\]

which proves~\eqref{eq:robustness}.

\end{proof}

Combining the robustness estimate above with the Monte Carlo approximation error yields the overall estimation error.

\begin{theorem}[Error Decomposition]
\label{thm:error}

Under the assumptions of
Proposition~\ref{prop:robustness},
there exists a constant
$C>0$
such that

\begin{equation}
\begin{aligned}
\left|
\Pi_t^N(u)
-
\pi_t^{\mathrm{true}}(u)
\right|
\le\;&
C
\int_0^t
\|
(I-P_{\parallel}(x_s))
\Delta f(x_s,s)
\|
ds
\\
&
+
C
\int_0^t
\|
\lambda_s
\|
ds
\\
&
+
\mathcal{O}(N^{-1/2}).
\end{aligned}
\label{eq:error_bound}
\end{equation}

\end{theorem}

\begin{proof}

The first two terms follow directly from the robustness estimate of
Proposition~\ref{prop:robustness}.

The final
$\mathcal{O}(N^{-1/2})$
term follows from the particle approximation result established in
Theorem~\ref{thm:particle_consistency}. Combining these bounds gives
\eqref{eq:error_bound}.

\end{proof}

\paragraph*{Interpretation.}

Equation~\eqref{eq:error_bound}
shows that the total estimation error consists of three complementary contributions. The first term represents the component of structural model mismatch orthogonal to the measurement-consistent tangent subspace. The second term is governed by the geometric co-state, which quantifies the residual incompatibility between the propagated dynamics and the instantaneous measurement geometry. The third term corresponds to the finite-sample approximation error inherent to Sequential Monte Carlo methods.

Consequently, the proposed Geometry-Consistent Particle Filter reduces sensitivity to structural model uncertainty by attenuating the mismatch before Bayesian weighting, thereby improving robustness while preserving the standard Monte Carlo convergence rate.

\section{From Continuous-Time Theory to Particle Filtering}
\label{sec:SGPF}

Sections~\ref{sec:GPD} and~\ref{sec:bayesian} established the geometry-consistent proposal dynamics and proved Bayesian posterior preservation. This section derives the corresponding discrete-time particle filter. The continuous-time proposal is first discretized, after which it is incorporated into the standard Sequential Monte Carlo framework. The resulting Geometric Projection Particle Filter (GPF) differs from the bootstrap particle filter only in the proposal stage, while Bayesian importance weighting and resampling remain unchanged.

\subsection{Discrete-Time Realization}
\label{sec:discretization}

The continuous-time geometry-consistent proposal derived in Section~\ref{sec:GPD} is implemented in discrete time using the Euler--Maruyama approximation. For a uniform discretization
$\{t_k\}_{k=0}^{K}$
with sampling interval
$\Delta t=t_{k+1}-t_k$,
the projected proposal evolves as

\begin{equation}
x_{k+1}
=
x_k
+
f_{\mathrm{proj}}(x_k,t_k)\,\Delta t
+
G(x_k,t_k)\,\Delta W_k,
\label{eq:discrete_proj}
\end{equation}

where
$f_{\mathrm{proj}}$
is the projected drift defined in~\eqref{eq:proj} and
$\Delta W_k\sim\mathcal{N}(0,\Delta t\,I)$
is the discrete Wiener increment.

Under the regularity assumptions of Section~\ref{sec:GPD}, the Euler--Maruyama approximation converges weakly to the continuous-time proposal as
$\Delta t\rightarrow0$.
Since the discretization modifies only the proposal dynamics, Bayesian posterior estimation remains unchanged and is recovered through the standard importance-weighting procedure established in Section~\ref{sec:bayesian}. Consequently, the discrete-time realization preserves the probabilistic consistency of the continuous-time formulation while introducing only the standard numerical discretization error associated with the Euler--Maruyama approximation.
\subsection{Geometric Projection Particle Filter}
\label{sec:gpf}

The proposed Geometric Projection Particle Filter (GPF) is obtained by incorporating the geometry-consistent proposal~\eqref{eq:discrete_proj} into the standard Sequential Monte Carlo framework. Relative to the classical bootstrap particle filter, the proposed algorithm modifies only the proposal propagation, while Bayesian importance weighting, normalization, and resampling remain unchanged.

The normalized importance weights evolve according to

\begin{equation}
w_{k+1}^{(i)}
\propto
w_k^{(i)}
\,
\frac{
p(y_{k+1}\mid x_{k+1}^{(i)})
\,p(x_{k+1}^{(i)}\mid x_k^{(i)})
}{
q(x_{k+1}^{(i)}\mid x_k^{(i)},y_{k+1})
},
\label{eq:weight_update}
\end{equation}

where
$q(\cdot)$
denotes the proposal distribution.

For the bootstrap particle filter,

\[
q(x_{k+1}\mid x_k,y_{k+1})
=
p(x_{k+1}\mid x_k),
\]

so particle propagation depends solely on the assumed system dynamics. Under structural model uncertainty, the proposal progressively departs from the posterior distribution, increasing importance-weight dispersion and accelerating particle degeneracy.

The proposed GPF replaces the nominal proposal with the geometry-consistent projected dynamics, thereby improving proposal quality before Bayesian importance weighting. Consequently, the dispersion of the importance weights is reduced,

\begin{equation}
\operatorname{Var} \!\left(\log w^{\mathrm{GPF}} \right)
\lesssim
\operatorname{Var} \!\left( \log w^{\mathrm{SPF}} \right),
\label{eq:weight_variance}
\end{equation}

where the remaining variability is governed by the residual dynamics--measurement inconsistency quantified by the geometric co-state,
$\lambda_t$.

The reduction in importance-weight dispersion increases the effective sample size,

\begin{equation}
\mathrm{ESS} = \frac{1} {\sum_{i=1}^{N} (\bar w^{(i)})^2},
\label{eq:ess}
\end{equation}

thereby preserving particle diversity and improving sampling efficiency under structural model uncertainty.

The resulting Geometric Projection Particle Filter therefore differs from the classical bootstrap particle filter only in the proposal construction, while preserving the standard Bayesian filtering formulation. The complete implementation is summarized in Algorithm~1, and its theoretical predictions are validated experimentally in Section~\ref{sec:results}.
\subsection{Computational Complexity}
\label{sec:complexity}

The proposed Geometric Projection Particle Filter (GPF) differs from the classical bootstrap particle filter only in the proposal stage. Let $N$, $n$, and $m$ denote the number of particles, state dimension, and measurement dimension, respectively.

The bootstrap particle filter scales linearly with the number of particles,
$\mathcal{C}_{\mathrm{SPF}}=\mathcal{O}(N)$. Relative to SPF, GPF additionally evaluates the measurement Jacobian, computes the regularized Moore--Penrose pseudoinverse of $HH^\top\in\mathbb{R}^{m\times m}$, and projects the nominal dynamics, yielding $\mathcal{C}_{\mathrm{GPF}} = \mathcal{O}\!\left(N(m^3+nm)\right).$

For spacecraft navigation, the measurement dimension is typically small ($m\le3$). The projection therefore introduces only a constant additional cost, preserving the overall linear scaling, $ \mathcal{C}_{\mathrm{GPF}} = \mathcal{O}(N).$ The memory requirement is $\mathcal{M}=\mathcal{O}(Nn)$, with negligible additional storage for the projection matrices. Consequently, GPF preserves the linear scalability of the bootstrap particle filter while requiring only a constant additional cost for geometry-consistent proposal construction.
\section{Algorithmic Realization}
This section presents the discrete-time implementation of the Geometric Projection Particle Filter (GPF), directly realizing the projected proposal dynamics of Section~\ref{sec:GPD} together with the Bayesian importance-weighting procedure established in Section~\ref{sec:bayesian}.

\begin{algorithm}[t]
\caption{Geometric Projection Particle Filter (GPF)}
\begin{algorithmic}[1]
\State Initialize particles $\{x_0^{(i)}\}_{i=1}^N$ from the prior distribution
\State Initialize weights $w_0^{(i)} = 1/N$
\For{each time step $t_k$}
\For{each particle $i$}

\State Compute measurement Jacobian $H_{t_k}^{(i)}$

\State Form projection operators
\[
P_{\parallel}^{(i)} =
H_{t_k}^{(i)\top}\!\bigl(H_{t_k}^{(i)}H_{t_k}^{(i)\top}\bigr)^\dagger
H_{t_k}^{(i)}
\]

\State \textbf{Prediction:}
\[
x_{t_k}^{(i),-} = x_{t_{k-1}}^{(i)} + f(x_{t_{k-1}}^{(i)},t_{k-1})\,\Delta t
\]

\State Compute co-state
\[
\lambda_{t_k}^{(i)} =
(H_{t_k}^{(i)}H_{t_k}^{(i)\top})^\dagger
\left(\Delta y_k - \eta(x_{t_k}^{(i),-},t_k)\,\Delta t\right)
\]

\State \textbf{Projection:}
\[
x_{t_k}^{(i)} =
x_{t_k}^{(i),-}
+ P_{\parallel}^{(i)}\!\big(\Delta y_k - \eta(x_{t_k}^{(i),-},t_k)\,\Delta t\big)
+ \sqrt{\Delta t}\, G(x_{t_k}^{(i),-},t_k)\,\xi_k^{(i)}
\]

\State Update weights:
\[
\log w_{k+1}^{(i)} =
\log w_k^{(i)} + \eta(x_{t_k}^{(i)},t_k)^\top \Delta y_k
- \tfrac12 \|\eta(x_{t_k}^{(i)},t_k)\|^2 \Delta t
\]

\EndFor

\State Normalize weights $\bar w_{t_k}^{(i)}$
\State Compute effective sample size (ESS) as defined in~\eqref{eq:ess}

\If{$\mathrm{ESS}(t_k) < \tau N$}
\State Resample particle set
\EndIf

\State Compute ensemble statistics of $\{\lambda_{t_k}^{(i)}\}$

\EndFor
\end{algorithmic}
\end{algorithm}

Algorithm~1 summarizes the implementation of the Geometric Projection Particle Filter (GPF). The algorithm realizes the projected proposal dynamics developed in Section~\ref{sec:GPD} together with the Bayesian importance-weighting procedure established in Section~\ref{sec:bayesian}. Since particles evolve independently, the implementation is naturally parallelizable.

The projected proposal is constructed from the measurement Jacobian and its regularized Moore--Penrose pseudoinverse. In practice, Tikhonov regularization is employed whenever
$H_{t_k}^{(i)}H_{t_k}^{(i)\top}$ becomes ill-conditioned, ensuring numerical robustness under partial observability and near-singular measurement configurations.

The geometric co-state $\lambda_t$ is evaluated independently for each particle as a diagnostic measure of dynamics--measurement consistency. Although it does not modify the Bayesian weighting or resampling stages, it provides an intrinsic indicator of estimator consistency throughout the filtering process.

\paragraph*{Implementation Remarks.}

The proposed algorithm differs from the bootstrap particle filter only in the proposal stage. Particles are first propagated using the nominal dynamics and then projected onto the measurement-consistent tangent subspace before Bayesian importance weighting. Likelihood evaluation, effective sample size computation, and resampling remain identical to the standard Sequential Monte Carlo framework, allowing straightforward integration into existing particle-filter implementations.
    
\begin{figure}[t]
\centering
\includegraphics[width=0.49\linewidth]{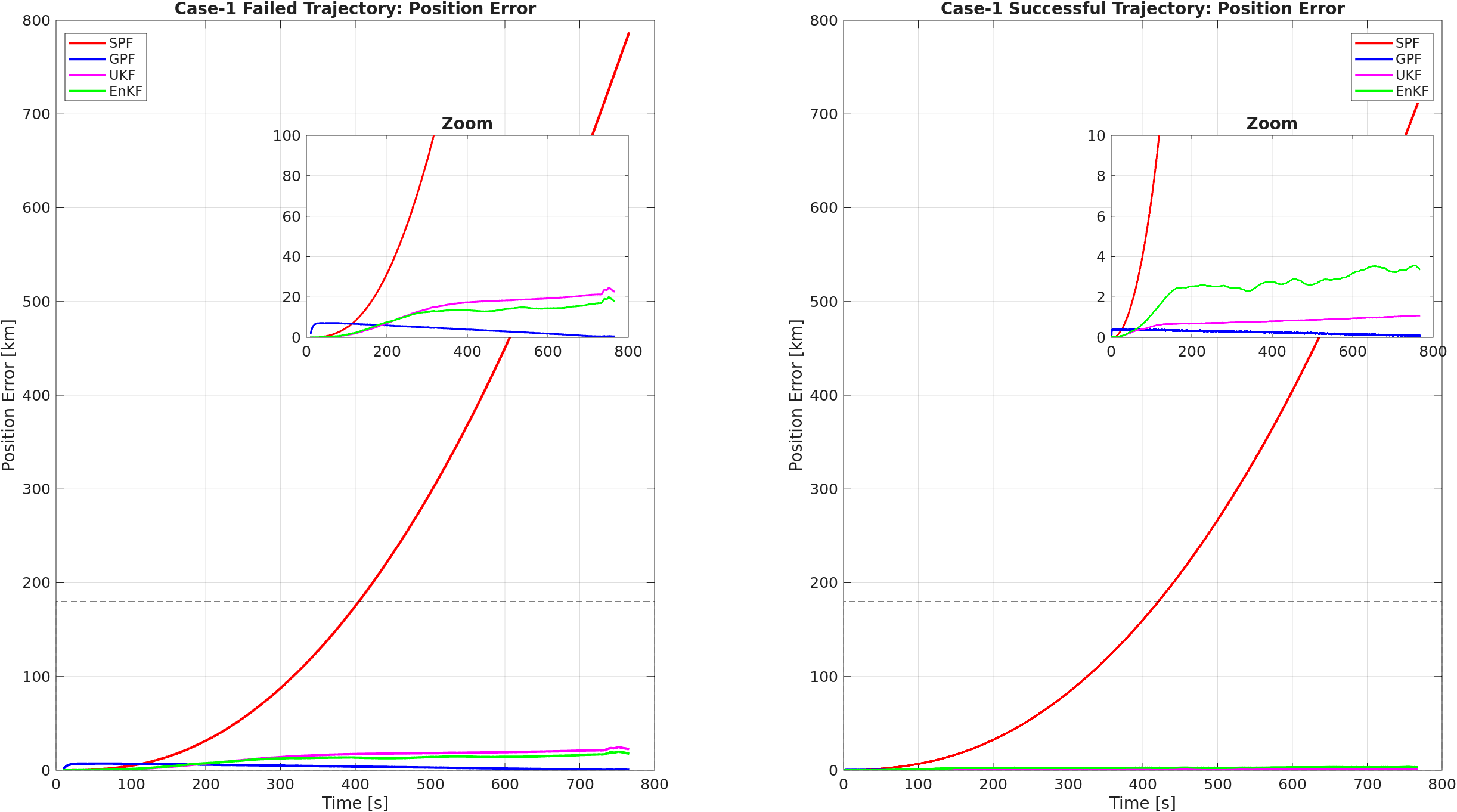}
\includegraphics[width=0.49\linewidth]{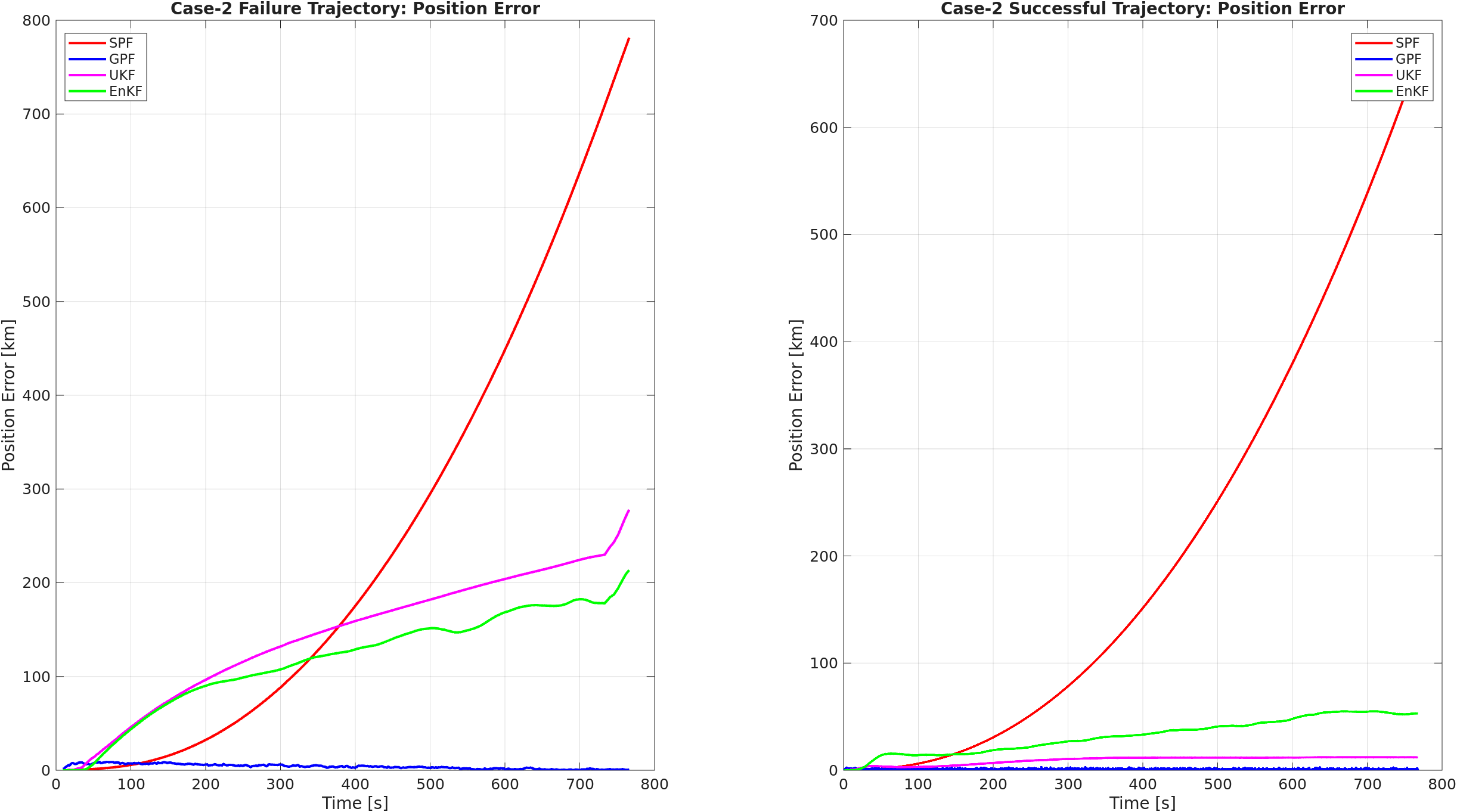}
\caption{
Estimation error over time for Case-1 (full observability, left) and Case-2 (partial observability, right). SPF exhibits rapid divergence, while EKF, UKF, and EnKF degrade under model mismatch. GPF maintains bounded error in both cases.
}
\label{fig:error_evolution}
\end{figure}

\begin{figure}[t]
\centering
\includegraphics[width=0.49\linewidth]{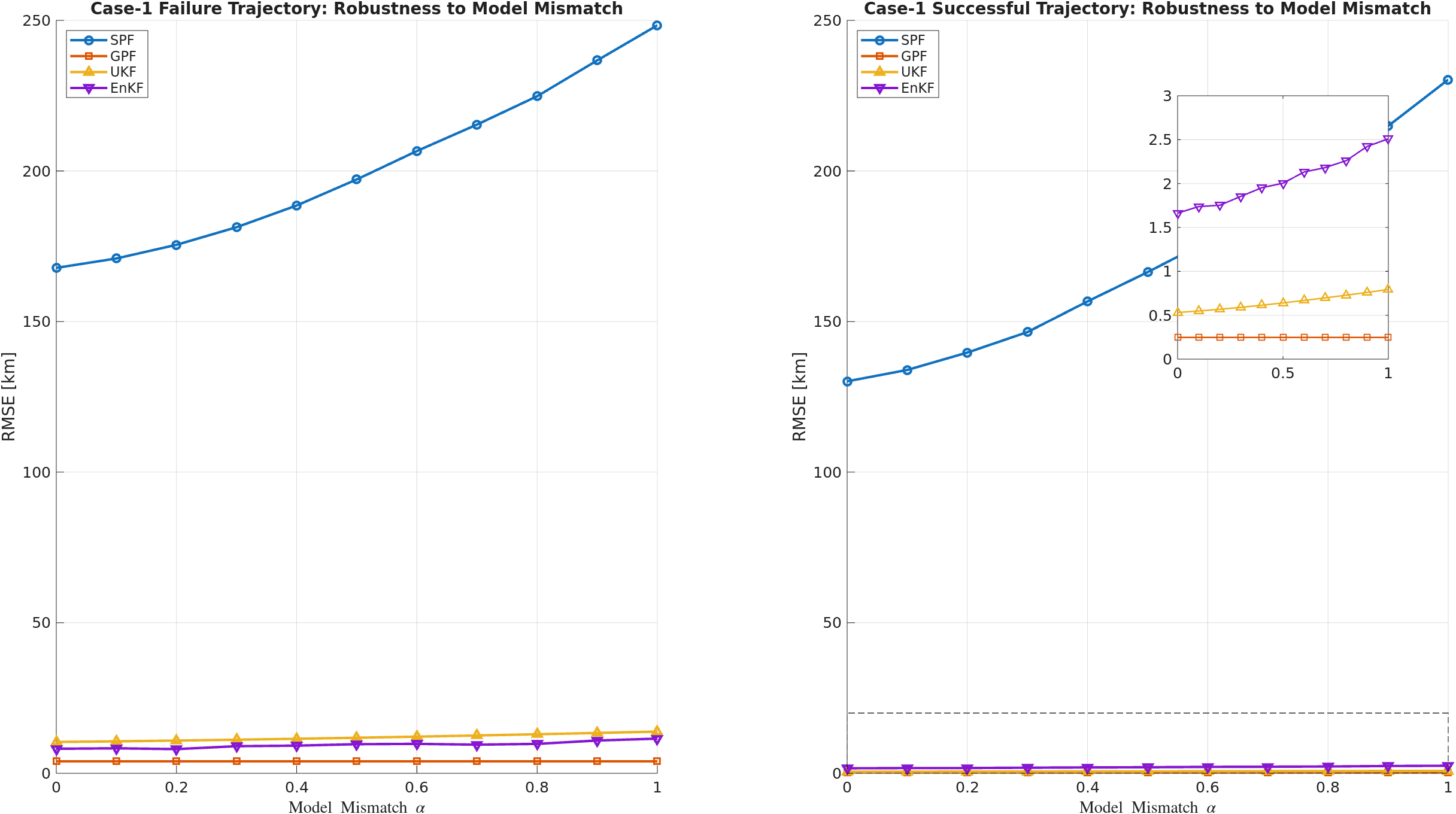}
\includegraphics[width=0.49\linewidth]{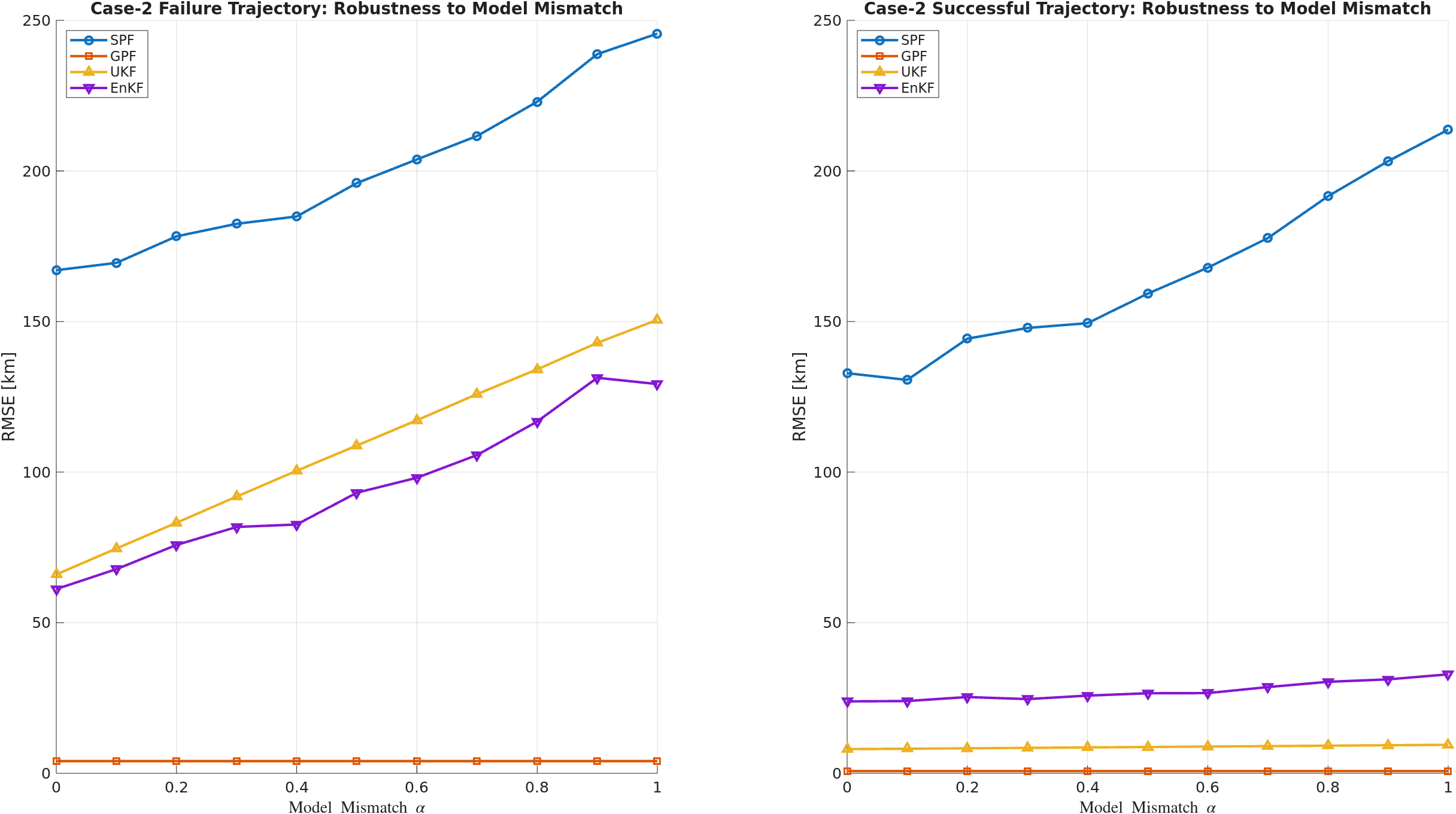}
\caption{
Estimation error versus model mismatch magnitude for Case-1 and Case-2. Classical filters degrade significantly as mismatch increases, whereas GPF remains stable.
}
\label{fig:model_mismatch}
\end{figure}

\begin{figure}[t]
\centering
\includegraphics[width=0.49\linewidth]{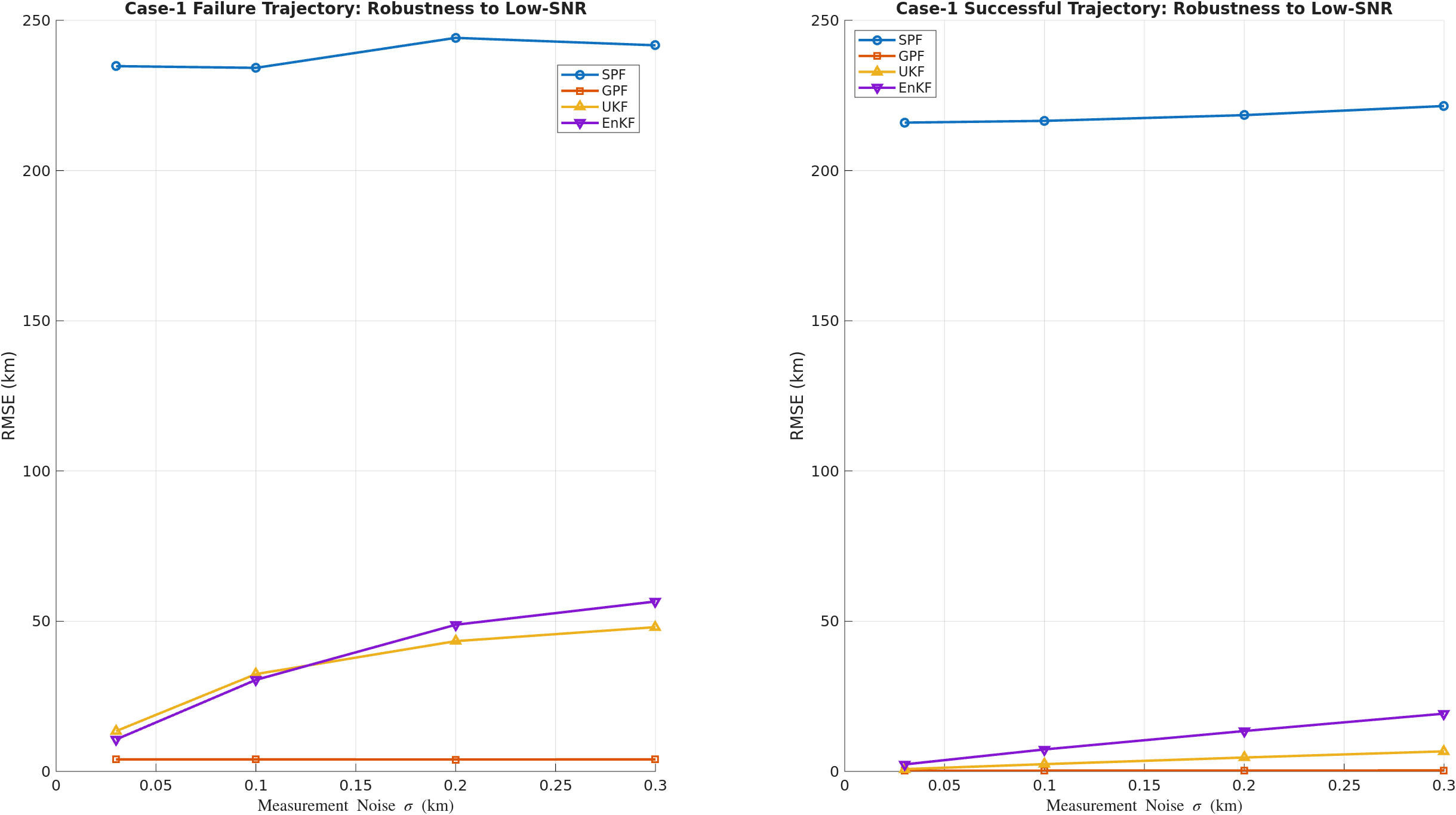}
\includegraphics[width=0.49\linewidth]{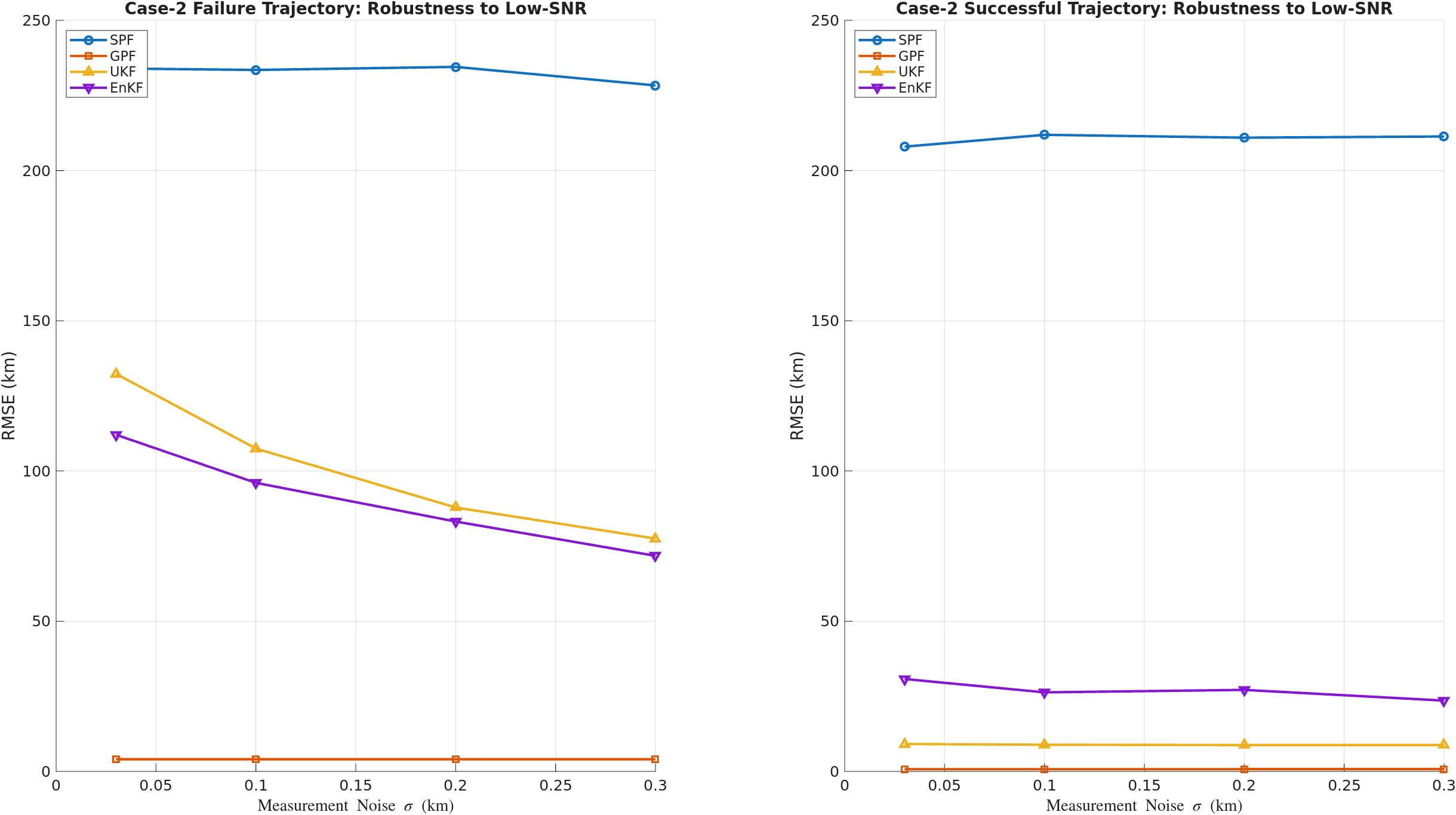}
\caption{
Estimation error versus measurement noise level for Case-1 and Case-2. Classical filters degrade with increasing noise, while GPF remains robust.
}
\label{fig:snr}
\end{figure}

\begin{figure}[t]
\centering
\includegraphics[width=0.49\linewidth]{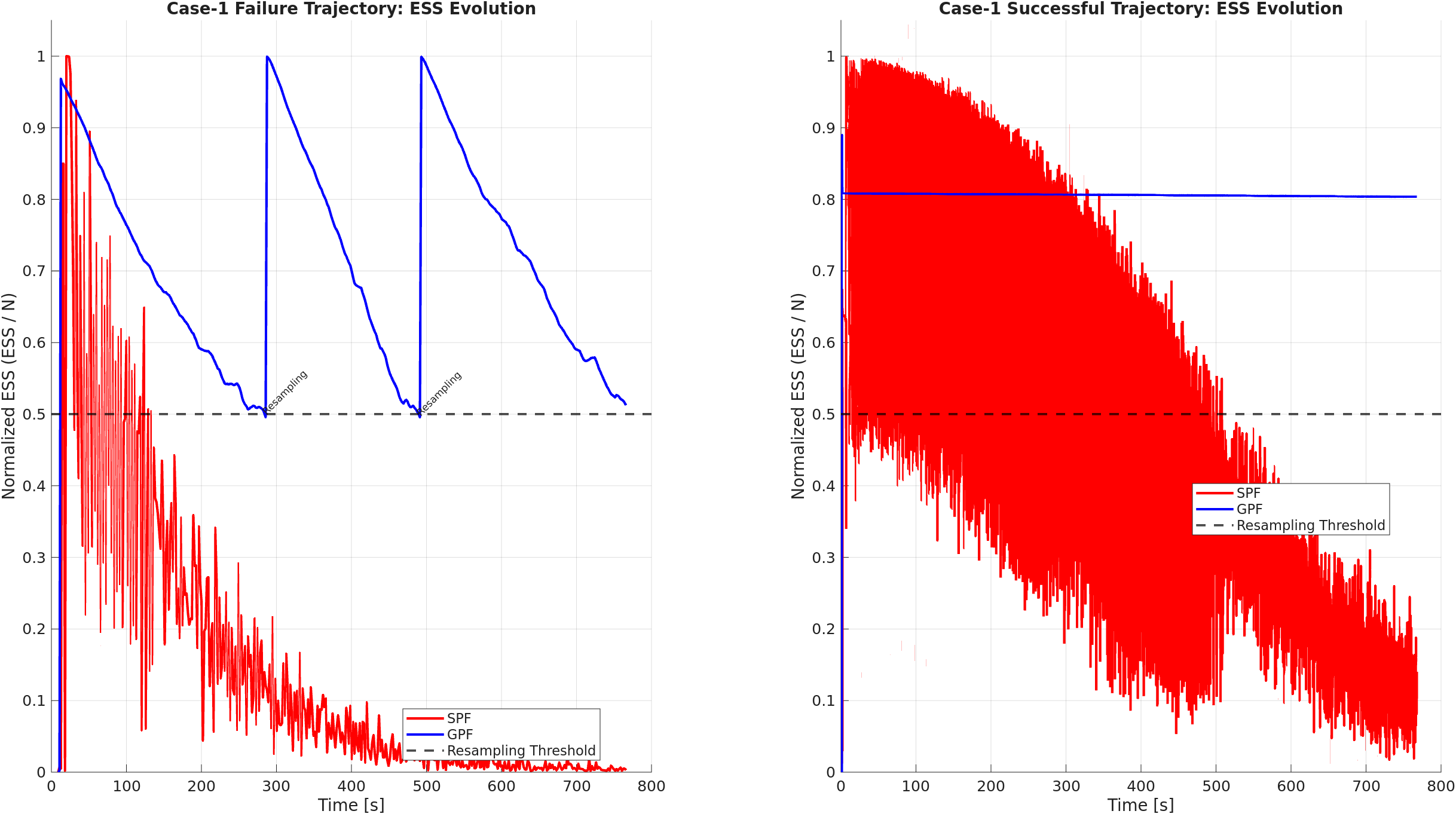}
\includegraphics[width=0.49\linewidth]{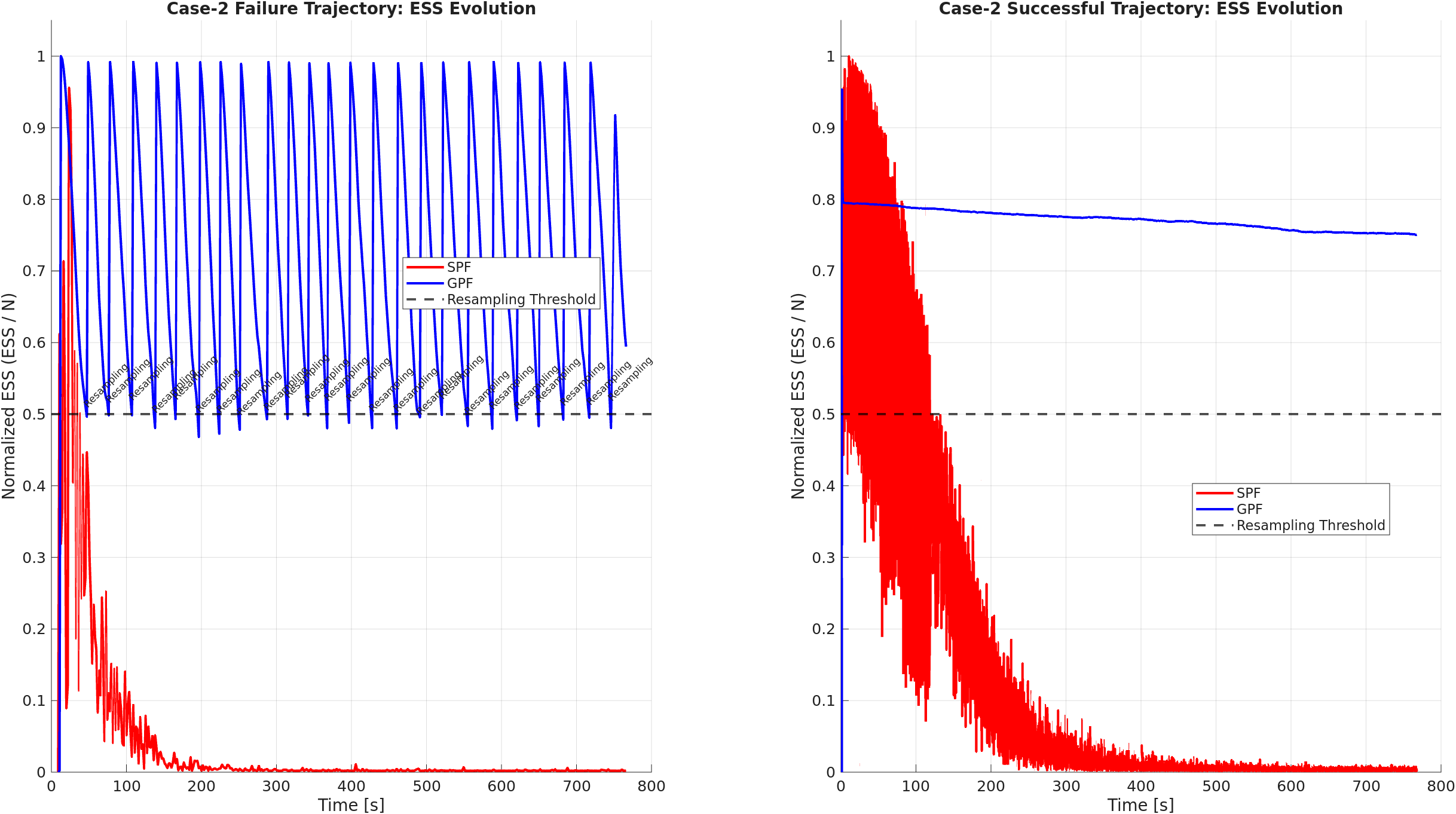}
\caption{
Normalized effective sample size (ESS) for Case-1 and Case-2. SPF exhibits rapid degeneracy, while GPF maintains high ESS, indicating improved weight stability.
}
\label{fig:ess}
\end{figure}

\begin{figure}[t]
\centering
\begin{minipage}{0.49\textwidth}
\centering
\includegraphics[width=\linewidth]{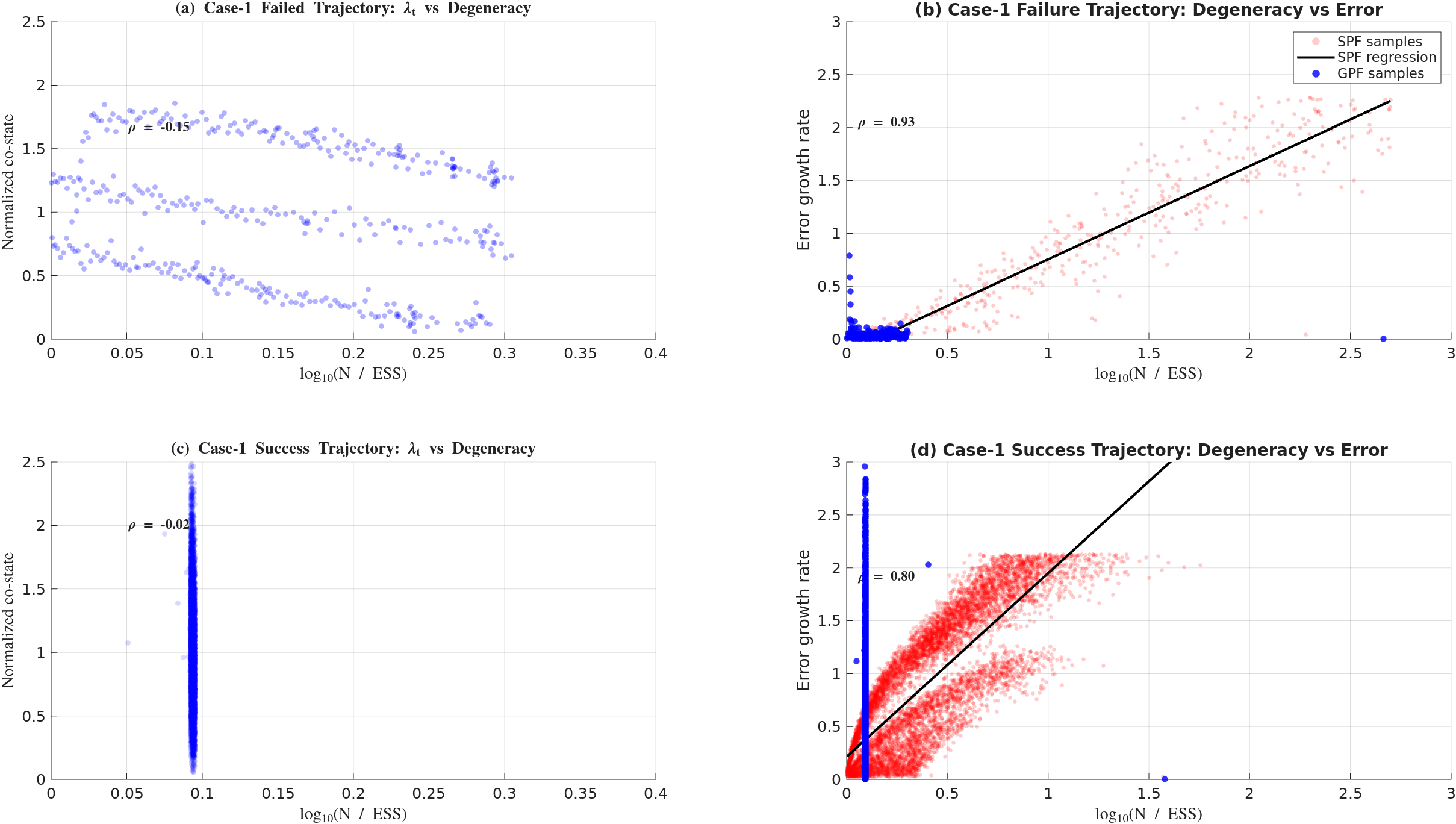}
\end{minipage}
\begin{minipage}{0.49\textwidth}
\centering
\includegraphics[width=\linewidth]{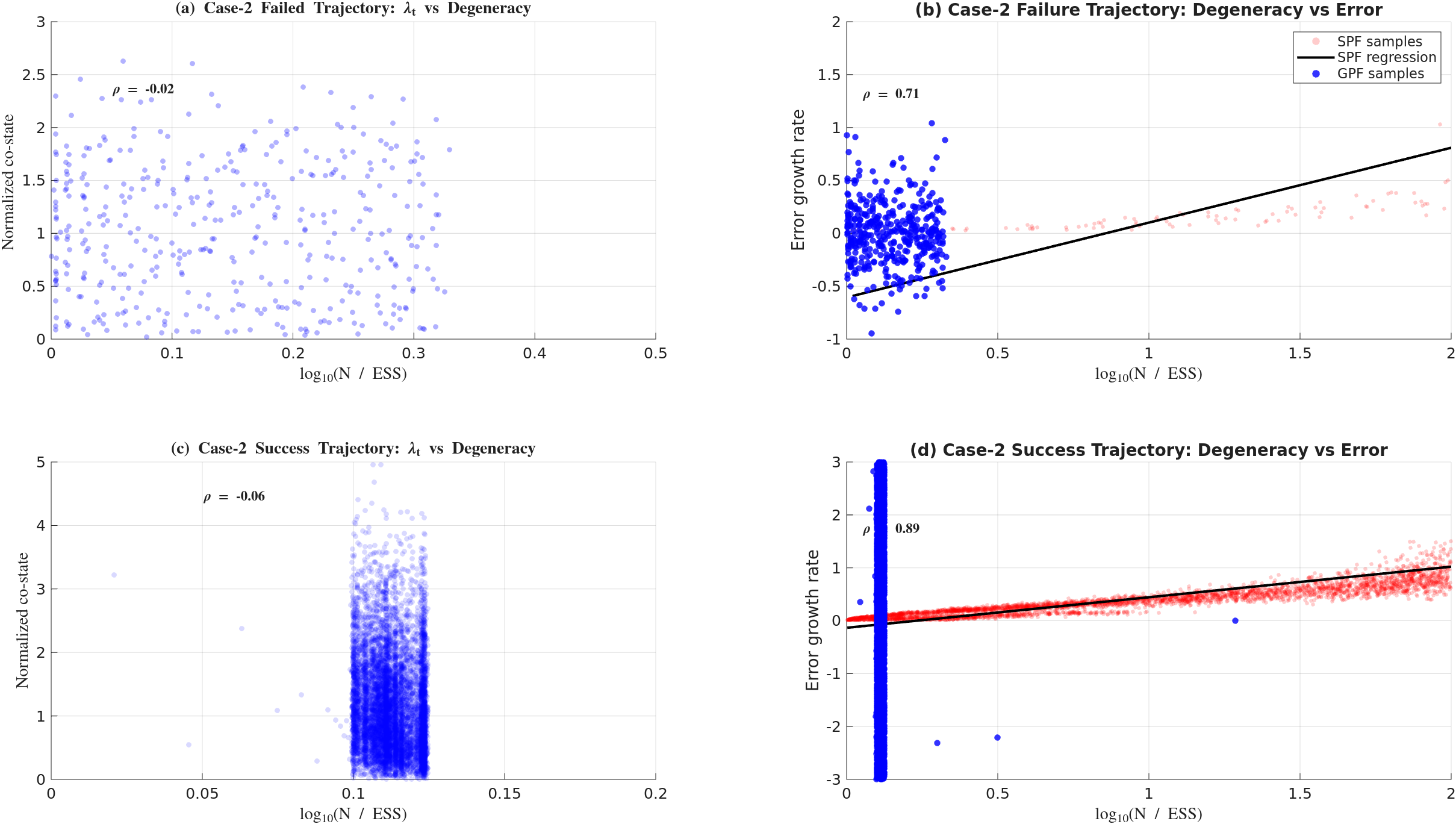}
\end{minipage}
\caption{
Relationship between co-state magnitude and particle degeneracy for Case-1 and Case-2. SPF exhibits strong coupling between degeneracy and error growth, whereas GPF remains in low-degeneracy regimes.
}
\label{fig:degeneracy_growth}
\end{figure}

\section{Application Framework and Results}
\label{sec:results}
This section evaluates the Geometric Projection Particle Filter (GPF) using high-fidelity lunar descent navigation data under structural model uncertainty, partial observability, and measurement noise.

The experiments assess the theoretical predictions developed in Sections~\ref{sec:GPD}--\ref{sec:gpf} by examining the relationship between geometric inconsistency, particle degeneracy, and estimation error. GPF is compared with the bootstrap Particle Filter (SPF), Extended Kalman Filter (EKF), Unscented Kalman Filter (UKF), and Ensemble Kalman Filter (EnKF), representing established sampling-based and Gaussian nonlinear filtering methods. The results demonstrate improved particle diversity, estimation accuracy, and robustness under structural model uncertainty.
\subsection{Experimental Setup}

The proposed Geometric Projection Particle Filter (GPF) is evaluated using high-fidelity lunar descent navigation data to assess estimation performance under structural model uncertainty and partial observability. The spacecraft state is

\[
x(t)=
\begin{bmatrix}
r(t)\\
v(t)
\end{bmatrix}
=
[x\;y\;z\;v_x\;v_y\;v_z]^\top,
\]

where $r(t)\in\mathbb{R}^3$ and $v(t)\in\mathbb{R}^3$ denote the position and velocity vectors, respectively. The nominal dynamics follow the two-body lunar gravitational model,

\[
\dot r=v,\qquad
\dot v=-\mu\frac{r}{\|r\|^3},
\]

with lunar gravitational parameter $\mu=4902.8~\mathrm{km}^3\mathrm{s}^{-2}$.

Structural model uncertainty is introduced by perturbing the assumed gravitational parameter,

\[
\dot v=
-\mu(1+\epsilon_{\mathrm{dyn}})
\frac{r}{\|r\|^3},
\qquad
\epsilon_{\mathrm{dyn}}=0.09,
\]

corresponding to a persistent $9\%$ dynamics mismatch between the assumed and true models.

GPF is compared with the bootstrap Particle Filter (SPF), Extended Kalman Filter (EKF), Unscented Kalman Filter (UKF), and Ensemble Kalman Filter (EnKF). All particle-based methods employ $N=500$ particles with systematic resampling whenever $\mathrm{ESS}<N/2$. Range–azimuth–elevation measurements are corrupted by zero-mean Gaussian noise with standard deviation $\sigma=0.03~\mathrm{km}$, and all filters are evaluated using identical initial conditions, measurement realizations, and process noise to ensure a fair comparison.

Two representative lunar descent scenarios are considered. Case~1 represents a mismatch-dominated trajectory in which structural model uncertainty leads to estimator degradation, whereas Case~2 represents a nominal descent under partial observability. Together, these scenarios evaluate estimation accuracy, particle diversity, and robustness across representative operating conditions.





\subsection{Experimental Validation Framework}
\label{sec:theory_results_link}

The theoretical developments of Sections~\ref{sec:GPD}--\ref{sec:gpf} predict that the principal benefit of geometry-consistent propagation is the suppression of proposal--posterior mismatch before Bayesian importance weighting. By projecting the nominal dynamics onto the measurement-consistent subspace, the proposed framework reduces the residual geometric inconsistency represented by the co-state $\lambda_t$. Consequently, particle-weight dispersion is reduced, particle diversity is preserved, and estimation accuracy is improved.

The resulting validation hypothesis is summarized by the following mechanism,

\begin{equation}
\text{Geometric Projection}
\;\longrightarrow\;
\|\lambda_t\|
\;\longrightarrow\;
\operatorname{Var}(\log w_t)
\;\longrightarrow\;
\mathrm{ESS}
\;\longrightarrow\;
\mathrm{RMSE},
\label{eq:validation_chain}
\end{equation}

where smaller geometric inconsistency leads to lower importance-weight variance, higher effective sample size, and reduced estimation error.

The following experimental results examine each stage of this mechanism. Estimation accuracy is first evaluated through trajectory and Monte Carlo analyses, followed by an assessment of particle diversity using the normalized effective sample size. The evolution of the geometric co-state is then analyzed to establish the relationship between projected model mismatch, particle degeneracy, and estimation performance under structural model uncertainty.
\subsection{Trajectory Estimation Accuracy}

Figure~\ref{fig:error_evolution} illustrates the temporal evolution of the position estimation error for all competing filters under full and partial observability. The objective is to evaluate the ability of each estimator to maintain trajectory accuracy in the presence of persistent structural model uncertainty.

For the full-observability scenario (Case~1), SPF exhibits rapid divergence after a short transient, with the estimation error exceeding $300~\mathrm{km}$. Although EKF, UKF, and EnKF remain numerically stable, all three accumulate increasing estimation bias as the model mismatch propagates through the prediction stage. In contrast, GPF maintains a bounded estimation error below approximately $5~\mathrm{km}$ throughout the descent, indicating that geometry-consistent propagation effectively limits the accumulation of prediction error.

The degradation becomes more pronounced under partial observability (Case~2). The reduction in measurement information accelerates the divergence of SPF and significantly increases the estimation error of the Gaussian filtering methods. Despite the reduced observability, GPF maintains stable trajectory tracking with only a modest increase in estimation error, demonstrating robustness to the combined effects of structural model uncertainty and limited measurement information.

These results establish that geometry-consistent proposal propagation maintains accurate state estimation under both nominal and degraded sensing conditions. The mechanisms responsible for this improved behavior are examined in the following subsections through analyses of particle diversity and geometric consistency.

\subsection{Monte Carlo Estimation Performance}
\label{sec:mc_performance}

To assess statistical robustness, all filters are evaluated over 100 Monte Carlo realizations using identical initial conditions, process noise, and measurement realizations. Table~\ref{tab:mc_results} summarizes the position RMSE (mean $\pm$ standard deviation) for both evaluation scenarios.

\begin{table}[t]
\centering
\caption{Monte Carlo estimation performance (RMSE in km, mean $\pm$ standard deviation).}
\label{tab:mc_results}
\begin{tabular}{lcc}
\hline
\textbf{Method} & \textbf{Full Observability} & \textbf{Partial Observability} \\
\hline
SPF   & $331.19 \pm 1.25$  & $250.40 \pm 3.60$ \\
GPF   & $\mathbf{4.52 \pm 0.00}$ & $\mathbf{5.28 \pm 0.22}$ \\
EKF   & $15.36 \pm 2.88$   & $111.22 \pm 43.70$ \\
UKF   & $15.31 \pm 0.00$   & $105.01 \pm 0.28$ \\
EnKF  & $12.11 \pm 0.45$   & $92.31 \pm 4.09$ \\
\hline
\end{tabular}
\end{table}

Across both evaluation scenarios, GPF consistently achieves the lowest mean RMSE and the smallest variability among all competing filters. Under full observability, the mean RMSE decreases from $331.19~\mathrm{km}$ (SPF) to $4.52~\mathrm{km}$, corresponding to an improvement of approximately two orders of magnitude. Under partial observability, GPF maintains comparable accuracy ($5.28~\mathrm{km}$), whereas the estimation error of SPF, EKF, UKF, and EnKF increases substantially.

The consistently small standard deviations indicate that the proposed geometry-consistent proposal yields stable estimation performance across independent Monte Carlo realizations rather than isolated favorable trajectories. This demonstrates that the observed improvement is both accurate and statistically repeatable. The paired Wilcoxon signed-rank test further confirms the statistical significance of the improvement ($p<0.001$) for both evaluation scenarios, indicating that the reduction in estimation error is unlikely to arise from Monte Carlo variability.
\subsection{Particle Degeneracy and Effective Sample Size}
\label{sec:ess}
Figure~\ref{fig:ess} evaluates particle diversity through the normalized effective sample size (ESS$/N$), which quantifies the proportion of particles contributing effectively to the posterior approximation. Since particle degeneracy is a primary limitation of the bootstrap particle filter, sustained ESS directly reflects the quality of the proposal distribution.

For the mismatch-dominated trajectory, SPF undergoes rapid weight collapse, with the normalized ESS decreasing below $0.1$ within approximately $200~\mathrm{s}$ and remaining close to zero thereafter. In contrast, GPF maintains ESS above $0.7$ throughout the descent, indicating that the proposed geometry-consistent propagation preserves particle diversity despite persistent structural model uncertainty. The corresponding time-averaged normalized ESS values are

\[
\frac{\mathrm{ESS}_{\mathrm{SPF}}}{N}\approx0.12,
\qquad
\frac{\mathrm{ESS}_{\mathrm{GPF}}}{N}\approx0.79.
\]

The sustained increase in ESS indicates that the projected proposal remains substantially closer to the posterior distribution than the bootstrap proposal. Consequently, importance weights remain more uniformly distributed across the particle ensemble, reducing the frequency of particle collapse and preserving sampling efficiency throughout the trajectory.

These observations provide the experimental mechanism underlying the RMSE improvements reported in Sections~\ref{sec:mc_performance} and~\ref{sec:results_discussion}: improved proposal quality leads to higher particle diversity, which in turn enables more accurate Bayesian state estimation under structural model uncertainty.
\subsection{Co-State, Degeneracy, and Error Growth Coupling}
\label{sec:degeneracy_growth}
Figure~\ref{fig:degeneracy_growth} validates the theoretical error-propagation mechanism proposed in Section~\ref{sec:theory_results_link}, namely that the geometric co-state governs particle degeneracy and the resulting estimation error. Subplots (a) and (c) relate the co-state magnitude, $\|\lambda_t\|$, to particle degeneracy. SPF exhibits a broad, highly dispersed distribution, indicating strong residual dynamics--measurement inconsistency and rapid particle collapse. In contrast, GPF confines $\|\lambda_t\|$ to a bounded region while maintaining high effective sample size, demonstrating that geometric projection suppresses proposal mismatch before Bayesian weighting.

Subplots (b) and (d) relate particle degeneracy to estimation error. For SPF, estimation error increases strongly with degeneracy ($\rho\approx0.8$--$0.9$), confirming that weight collapse directly amplifies filtering error. In contrast, GPF maintains low estimation error over the observed degeneracy range, indicating a substantially weaker dependence between particle collapse and estimation accuracy. These results validate the proposed error-propagation mechanism: reducing geometric inconsistency limits importance-weight degeneracy and consequently suppresses the propagation of structural model mismatch into the state estimate.
\subsection{Robustness to Structural Model Uncertainty}
\label{sec:robustness}

Figure~\ref{fig:model_mismatch} evaluates filter performance under progressively increasing structural model uncertainty, providing experimental validation of Proposition~\ref{prop:robustness}, which predicts that the filtering error is governed by the projected mismatch $(I-P_{\parallel})\Delta f$. As the model mismatch increases, SPF exhibits rapid error growth because propagation relies entirely on the nominal dynamics. EKF, UKF, and EnKF show similar degradation, with estimation accuracy deteriorating monotonically as prediction errors accumulate.

In contrast, GPF maintains stable performance across the entire mismatch range, with only a modest RMSE increase (approximately $10$--$15\%$). This behavior indicates that the geometry-consistent projection suppresses the mismatch component orthogonal to the measurement-consistent subspace before Bayesian importance weighting, thereby limiting the propagation of structural errors into the posterior estimate. These results validate the proposed robustness mechanism: attenuating the projected model mismatch preserves proposal consistency and maintains accurate state estimation under persistent structural uncertainty.

\subsection{Robustness to Measurement Noise}

Figure~\ref{fig:snr} evaluates estimation performance under progressively increasing measurement noise to assess the robustness of the proposed geometry-consistent framework. As the noise level increases, SPF, EKF, UKF, and EnKF exhibit progressively larger estimation errors owing to reduced measurement reliability. In contrast, GPF maintains consistently lower estimation error across the entire noise range, demonstrating stable performance under degraded sensing conditions.

The improved robustness follows directly from the geometry-consistent proposal. By enforcing compatibility with the local measurement geometry prior to Bayesian importance weighting, the projected proposal remains closer to the posterior distribution, reducing proposal--posterior mismatch and preserving sampling efficiency despite elevated measurement noise. These results demonstrate that geometry-consistent propagation improves robustness to measurement uncertainty while preserving the underlying Bayesian filtering formulation.
\subsection{Discussion of Results}
\label{sec:results_discussion}

The experimental results consistently validate the theoretical framework developed in Sections~\ref{sec:GPD}--\ref{sec:gpf}. Across all evaluation scenarios, geometry-consistent proposal propagation reduces the projected model mismatch, suppresses importance-weight dispersion, preserves effective sample size, and consequently improves estimation accuracy. The observed relationship,

\[
(I-P_{\parallel})\Delta f
\;\rightarrow\;
\|\lambda_t\|
\;\rightarrow\;
\operatorname{Var}(\log w_t)
\;\rightarrow\;
\mathrm{ESS}
\;\rightarrow\;
\mathrm{RMSE},
\]

is consistently supported by the experimental evidence.

Unlike conventional particle filtering, which addresses model mismatch only through importance weighting, GPF reduces the inconsistency during proposal propagation while preserving the Bayesian posterior through the subsequent importance-weight correction. Consequently, the proposed framework achieves robust nonlinear state estimation under structural model uncertainty without modifying the underlying Bayesian filtering formulation.
\section{Conclusion}

This paper presented the Geometric Projection Particle Filter (GPF) for nonlinear state estimation under structural model uncertainty and partial observability. By incorporating measurement geometry directly into particle propagation through projection onto the measurement-consistent subspace, the proposed approach reduces proposal--likelihood mismatch while preserving the Bayesian posterior via a rigorous change of measure formulation.

Theoretical results established existence and uniqueness of the projected dynamics, Bayesian posterior preservation, standard Monte Carlo convergence of the particle approximation, and robustness to structural model uncertainty by showing that the filtering error is governed by the component of model mismatch orthogonal to the measurement-consistent subspace. The geometric co-state was shown to provide an intrinsic measure of dynamics--measurement inconsistency, linking structural model mismatch with particle degeneracy and estimation error.

Numerical experiments on lunar descent navigation demonstrated consistently higher effective sample size, reduced particle degeneracy, and substantially improved estimation accuracy compared with conventional particle filtering under persistent model mismatch and partial observability. These results demonstrate that incorporating measurement geometry into particle propagation provides a principled, computationally efficient, and robust framework for nonlinear Bayesian filtering.

Future work will investigate adaptive projection strategies, high-dimensional filtering, and integration with autonomous guidance, navigation, and control architectures.
	\bibliographystyle{elsarticle-num}
	\bibliography{references_v4}	


\end{document}